\documentclass[11pt]{amsart}

\usepackage{amscd,amsmath,amssymb,fancyhdr,color}
\usepackage{stackengine}
\usepackage{mathtools,bm}
\usepackage{tikz-cd}

\newcommand{\al}{\alpha}

\newcommand{\f}{\varphi}

\newcommand{\Ker}{\text{Ker}}

\stackMath
\newcommand\reallywidehat[1]{%
\savestack{\tmpbox}{\stretchto{%
  \scaleto{%
    \scalerel*[\widthof{\ensuremath{#1}}]{\kern-.6pt\bigwedge\kern-.6pt}%
    {\rule[-\textheight/2]{1ex}{\textheight}}
  }{\textheight}%
}{0.5ex}}%
\stackon[1pt]{#1}{\tmpbox}%
}
\newcommand{\SL}{\mathrm{SL}}

\numberwithin{equation}{section}

\def\eqref#1{(\ref{#1})}

\newcommand{\N}{{\mathbb N}}
\newcommand{\Z}{{\mathbb Z}}
\newcommand{\C}{{\mathbb C}}
\newcommand{\R}{{\mathbb R}}
\newcommand{\Q}{{\mathbb Q}}
\renewcommand{\H}{{\mathbb H}}

\def\1{\sqrt{-1}\:}

\newcommand{\cntrct}                
{\hspace{2pt}\raisebox{1pt}{\text{$\lrcorner$}}\hspace{2pt}}
\newcommand{\arrow}{{\:\longrightarrow\:}}

\renewcommand{\bar}{\overline}
\renewcommand{\phi}{\varphi}
\renewcommand{\epsilon}{\varepsilon}
\renewcommand{\geq}{\geqslant}

\newcommand{\Aut}{\operatorname{Aut}}

\renewcommand{\dim}{\operatorname{dim}}

\newcommand{\rk}{\operatorname{rk}}

\newcommand{\Nov}{\operatorname{Nov}}

\renewcommand{\Re}{\operatorname{Re}}
\renewcommand{\Im}{\operatorname{Im}}

\newcommand{\vol}{\operatorname{vol}}

\newcounter{Mycounter}[section]
\newcounter{lemma}[section]
\newcounter{claim}[section]

\newcounter{sublemma}[section]

\newcounter{corollary}[section]

\newcounter{theorem}[section]

\newcounter{conjecture}[section]

\newcounter{proposition}[section]

\newcounter{definition}[section]

\newcounter{example}[section]

\newcounter{remark}[section]

\newcounter{problem}[section]

\newcounter{question}[section]

\makeatletter

\@addtoreset{equation}{section}

\@addtoreset{footnote}{section}

\makeatother

\def\blacksquare{\hbox{\vrule width 5pt height 5pt depth 0pt}}

\begin{document}

\newpage

\title[Morse-Novikov cohomology of closed one--forms]{Morse-Novikov cohomology of closed one--forms of rank 1}
\author{Alexandra Otiman}

\date{\today}
\address{Institute of Mathematics ``Simion Stoilow'' of the Romanian Academy\\
21, Calea Grivitei Street, 010702, Bucharest, Romania {\em and}\newline University of Bucharest, Faculty of Mathematics and Computer Science, 14 Academiei Str., Bucharest, Romania.}
\email{alexandra\_otiman@yahoo.com}

\abstract  We discuss the Morse-Novikov cohomology of a compact manifold, associated to a closed one--form whose free abelian group generated by its periods $\langle \int_\gamma \eta \mid [\gamma] \in \pi_1(M)\rangle$ is of rank 1, the focus being on locally conformally symplectic manifolds. In particular, we provide an explicit computation for the Inoue surface $\mathcal{S}^0$. \\[.1in]

\endabstract

\maketitle

\tableofcontents

\section{Introduction}

The Morse-Novikov cohomology of a manifold $M$ refers to the cohomology of the complex of smooth real forms $\Omega^{\bullet}(M)$, with the differential operator perturbed with a closed one-form $\eta$, defined as follows
\begin{equation}
d_{\eta} := d - \eta \wedge \cdot
\end{equation}
Indeed, the closedness of $\eta$ implies $d_{\eta}^2=0$, whence $d_{\eta}$ produces a cohomology, which we denote by $H^{\bullet}_\eta(M)$. Throughout this paper, we shall use the name Morse-Novikov for the cohomology $H^{\bullet}_\eta(M)$, although the name Lichnerowicz cohomology is also used in the literature (see \cite{bk}, \cite{hr}). Its study began with Novikov (\cite{n1}, \cite{n2}) and was independently developed by Guedira and Lichnerowicz (\cite{gl}). 

The Morse-Novikov cohomology has more than one description.
To begin with, let us look at the following exact sequence of sheafs:
\begin{equation}\label{rezolutie}
0 \rightarrow \Ker\, d_\eta \xrightarrow{i}  \Omega^0_{M}( \cdot ) \xrightarrow{d_{\eta}} \Omega^1_{M}( \cdot )\xrightarrow{d_\eta}\Omega^2_{M}( \cdot )\xrightarrow{d_\eta} \cdots
\end{equation}
where we denote by $\Omega^k_{M}( \cdot )$ the sheaf of smooth real $k$-forms on $M$. In fact, the sequence above is an acyclic resolution for $\Ker\, d_\eta$, as each $ \Omega^i_{M}( \cdot )$ is soft (see Proposition 2.1.6 and Theorem 2.1.9 in \cite{d}). Thus, by taking global sections in \eqref{rezolutie}, we compute the cohomology groups of $M$ with values in the sheaf $\Ker\, d_\eta$, $H^i(M, \Ker\, d_\eta)$. What we obtain is actually the Morse-Novikov cohomology.

The sheaf $\Ker\, d_\eta$ has the property that there exists a covering $(U_i)_i$ of $M$, such that it is constant when restricted to each $U_i$. In order to see that, one simply takes a contractible covering $(U_i)_i$ for which $\eta=df_i{_{|U_i}}$, then by considering the map $g \mapsto e^{-f_i}g$, we get an isomorphism $\Ker\, d_{\eta}(U_i) \simeq \R$.

Moreover, the covering $(U_i)_i$ and the isomorphisms above associate to $\Ker\, d_\eta$ a line bundle $L$, which is trivial on this covering and whose transition maps are $g_{ij}=e^{f_i -f_j}$. It is immediate that $(U_i, e^{-f_i})$ defines a global nowhere vanishing section $s$ of $L^*$, which is the dual of $L$ and by means of $s$, $L^*$ is isomorphic to the trivial bundle. We define a flat connection $\nabla$ on $L^*$ by $\nabla s= - \eta \otimes s$. Then $H^i_\eta(M)$ can also be computed as the cohomology of the following complex of forms with values in $L^*$: 
\begin{equation}
0  \rightarrow  \Omega^0(M, L^* ) \xrightarrow{\nabla}  \Omega^1(M, L^* )\xrightarrow{\nabla}  \Omega^2(M, L^* ) \xrightarrow{\nabla} \cdots
\end{equation}

The Morse-Novikov cohomology is not a topological object in essence, however it can provide information about the closed one-form to which it is associated. For instance, it was shown in \cite{llmp} that if on a compact manifold $M$ there exists a Riemannian metric $g$ and a closed one-form $\eta$ such that $\eta$ is parallel with respect to $g$, then for any $i \geq 0$, $H^i_\eta(M)=0$.

Some properties verified by the Morse-Novikov cohomology are summarized in the following:
\begin{proposition} \label{prop1} Let $M$ be a $n$-dimensional manifold and $\eta$ a closed one-form. Then
\begin{enumerate} 
\item if $\eta'= \eta+ df$, for any $i \geq 0$, $H^i_{\eta'}(M) \simeq H^i_{\eta}(M)$ and the isomorphism is given by the map $[\alpha] \mapsto [e^{-f}\alpha]$.
\item (\cite{hr}, \cite{gl}) if $\eta$ is not exact and $M$ is connected and orientable, $H^0_\eta(M)$ and $H^n_\eta(M)$ vanish.
\item (\cite{bk}) the Euler characteristic of the Morse-Novikov cohomology coincides with the Euler characteristic of the manifold, as a consequence of the Atyiah-Singer index theorem, which implies that the index of the elliptic complex $(\Omega^k(M), d_\eta)$ is independent of $\eta$.
\end{enumerate}
\end{proposition} 
From now on, we shall assume throughout the paper that $M$ is a compact manifold, unless specified. 
We denote by $\chi: \pi_1(M) \rightarrow \R$ the morphism of periods of $\eta$, namely $[\gamma] \mapsto \int_{\gamma} \eta$. 
\begin{definition} The {\em rank} of $\eta$ is the rank of $\Im \chi$ as a free abelian group.
\end{definition}

The fundamental group of $M$ is finitely presented, hence $\Im \chi$ has finite rank and it is isomorphic to a free abelian group $\Z^r$.
The following characterization was proven in  \cite{ov}:

\begin{proposition}  The {\em rank} of a closed one-form $\eta$ is the dimension of the smallest rational subspace $V \subset H^1(M, \Q)$ such that $[\eta]$ lies in $V \otimes_{\Q} \R$.
\end{proposition}

In other words, the rank of $\eta$ is the maximum number of rationally independent periods.

\begin{remark}
If $\eta$ is a closed one-form of rank 1, then $\Im \chi = \alpha \cdot \Z$, where $\alpha$ is a real number. 
\end{remark}

Motivated by the natural setting that locally conformally symplectic manifolds provide for the Morse-Novikov cohomology, the aim of this note is to present some explicit examples and computations, of particular interest being the Inoue surface $\mathcal{S}^0$. 



The note is organized as follows. As we try to present this material as self-contained as possible, we give the necessary preliminaries in Section 2 and explain the tools we shall use in the sequel, namely a result of A. Pajitnov in \cite{paj}, which relates the so-called  {\em Novikov Betti numbers} to the Morse-Novikov cohomology and a twisted version of the Mayer-Vietoris sequence of S. Haller and T. Rybicki presented in \cite{hr}. Section 3 is devoted to introducing locally conformally symplectic manifolds and to computing the Morse-Novikov cohomology of the Inoue surface $\mathcal{S}^0$ with respect to the closed one-form that F.Tricerri proves in \cite{t} to be the Lee form of a locally conformally K\"ahler form.





\section {Preliminaries}

We first give some definitions, in order to state later the results we are going to use.

\begin{definition} Let $\Gamma \subset \R$ be a subgroup of $\R$. The {\em Novikov ring} associated to $\Gamma$ is defined as the following ring of formal sums:
$$\Nov(\Gamma)=\left\{\sum_{i=1}^{\infty}n_iT^{\gamma_i} \mid n_i \in \Z, \gamma_i \in \Gamma, \lim_{i \rightarrow \infty}\gamma_i=-\infty\right\}$$
\end{definition}

\begin{remark} The ring $Nov(\Gamma)$ is a principal ideal domain (see \cite[Lemma 1.15]{f}).
\end{remark}

In what follows, we consider $\Gamma$ to be the group of periods of $\eta$. In this case, there is a $\Z[\pi_1(M)]$-module structure of $\Nov(\Gamma)$ described by $[\gamma] \cdot n = T^{-\int_{\gamma}\eta} \cdot n$, for any element $n$ in $\Nov(\Gamma)$. This further describes a $\Nov(\Gamma)$-local system on the manifold $M$, which we denote by $\widetilde{\Nov(\Gamma)}$. We recall that if $R$ is a commutative ring, there is a correspondence between representations of the fundamental group $\rho : \pi_1(M) \rightarrow \Aut(R)$, $\Z[\pi_1(M)]$-module structures on $R$ and $R$-local systems.  For more details, see Proposition 2.5.1 and Chapter 2 in \cite{d}. 

\begin{definition} For any $i \in \N$, the $i$-th {\em Novikov Betti number} is
$$b_i^{Nov}(M) := \rk_{\Nov(\Gamma)} H_{i}(M,  \widetilde{\Nov(\Gamma)}).$$
\end{definition}

\begin{remark} As $M$ is compact,  $H_{i}(M,  \widetilde{\Nov(\Gamma)})$ is a finitely generated module over a principal ideal domain and hence it decomposes as a direct sum of a free and a torsion part. We mean by $\rk_{\Nov(\Gamma)} H_{i}(M,  \widetilde{\Nov(\Gamma)})$ the rank of the free part of $H_{i}(M,  \widetilde{\Nov(\Gamma)})$.
\end{remark}

The relation between the Morse-Novikov cohomology and the Novikov Betti numbers is given by the following result of A. Pajitnov, which we state in the form we shall need:

\begin{theorem}\label{mail}(\cite[Lemma 2]{paj}) Let $M$ be a manifold and $\eta$ a closed one-form of rank 1, such that $\Im \chi = \alpha \Z$, with $e^\alpha$ transcendental. Then for any $i \geq 0$:
$$b^{Nov}_i = \dim_{\R}H^i_{\eta}(M).$$
\end{theorem}
The importance of the Novikov Betti numbers is of topological nature and reffers to the extension of the classical Morse theory to closed one-forms developed by Novikov. His initial motivation was to find a relation between the zeros of a closed one-form $\eta$, {\em of Morse type} (namely, locally given by $\eta=df$, where $f$ is a Morse function) and the topology of the manifold. The tool he created is a complex $(N^k_{\eta*}, \delta_k)$, where $N^k_{\eta*}$ is a free $Nov(\Gamma)$-module, whose generators are in 1-1 correspondence with the zeros of index $k$ of $\eta$. For the differentials $\delta_k$, one needs to consider the cover $M_{\eta} \xrightarrow{\pi} M$ corresponding to the group $\Ker\, \chi$, which is the minimal cover on which $\pi^*\eta$ is exact, and count down flow lines of a Smale vector field between critical points of consecutive index of the primitive of $\pi^*\eta$ (which turns out to be a Morse function). For more details regarding this construction, see \cite{f}, \cite{l} and \cite{poz}.

The following result was stated by Novikov, \cite{n2}, but proven rigorously by F. Latour in \cite{l} and M. Farber \cite{f2}. 

\begin{theorem}\label{novikov} Let $M$ be a compact manifold, $\eta$ a closed one-form of Morse type and $(N_{\eta*}, \delta_*)$ the Novikov complex associated to $\eta$. Then:
$$H_{i}(N_{\eta*}, \delta_*) \simeq H_{i}(M,  \widetilde{\Nov(\Gamma)}).$$
\end{theorem}

Therefore, the relevance of the Novikov complex is that, like the Morse-Smale complex, it produces a topological result by using Morse theory. The complete proof of the theorem above can be found in  \cite[Chapter 3]{f}. 
\begin{remark} We notice that if $\eta$ is a nowhere vanishing closed one-form, then the homology of $(N_{\eta*}, \delta_*)$ is 0 and by \ref{novikov}, all $b_i^{Nov}$ vanish. If moreover, $\eta$ satisfies the conditions in \ref{mail}, the Morse-Novikov cohomology with respect to $\eta$ also vanishes.
\end{remark}

The second tool we shall use in this note, in order to compute the Morse-Novikov cohomology groups of $\mathcal{S}^0$ is the following version of the Mayer-Vietoris sequence:
  
\begin{lemma} (\cite[Lemma 1.2]{hr}) Let $M$ be the union of two open sets $U$ and $V$ and $\theta$ a closed one-form. Then there exists a long exact sequence
\begin{multline}
\cdots \rightarrow H^i_\theta(M)\xrightarrow{\alpha_*}H^i_{\theta_{|U}}(U) \oplus H^i_{\theta_{|V}}(V)\xrightarrow{\beta_*}\\
\xrightarrow{\beta_*}H^i_{\theta_{|U \cap V}}(U \cap V)\xrightarrow{\delta} H^{i+1}_{\theta}(M)\rightarrow \cdots
\end{multline}
where for some  partition of unity $\{\lambda_U, \lambda_V\}$ subordinated to the covering $\{U,V\}$, the above morphisms are:
\begin{equation*}
\begin{split}
\delta ([\sigma]) &= [d\lambda_U \wedge \sigma]=-[d\lambda_V \wedge \sigma],\\
\alpha (\sigma)&= (\sigma_{|U}, \sigma_{|V}),\\
\beta(\sigma, \tau)&=\sigma_{|U \cap V}-\tau_{|U \cap V}.
\end{split}
\end{equation*}
\end{lemma}

\begin{remark} In the case of the Inoue surface, we shall be interested in computing the Morse-Novikov cohomology of a closed one-form of rank 1, whose group of periods $\Gamma$ is $\alpha\Z$ such that $e^{\alpha}$ is algebraic. The author in \cite{paj} gives an explicit computation, which covers this situation, as well. However, we use the Mayer-Vietoris sequence approach, since it is more direct. 
\end{remark}

\section{LCS manifolds and Morse-Novikov cohomology of the Inoue surface $\mathcal{S}^0$}

{\em Locally conformally symplectic manifolds} (shortly LCS) are smooth real (necessarily even-dimensional) manifolds endowed with a nondegenerate two form $\omega$ which satisfies the equality
\begin{equation} \label{lcs} 
d \omega = \theta \wedge \omega
\end{equation}
for some closed one form $\theta$, called the {\em Lee form}. 

Equivalently, this means there exists a non-degenerate two-form $\omega$, a covering of the manifold, $\{U_i\}_i$ and smooth functions $f_i$ on $U_i$ such that $e^{-f_i}\omega$ are symplectic, which literally explains their name. 

The equality \eqref{lcs} rewrites as $d_\theta \omega=0$, hence the problem of studying on an LCS manifold the Morse-Novikov cohomology associated to the Lee form of an LCS structure is natural.  

 Contact geometry is a source of examples of LCS manifolds. We adopt the following:
\begin{definition} Let $M$ be a manifold of odd dimension $2n+1$. Then $M$ is a {\em contact manifold} if there exists a one-form $\alpha$ such that $\alpha \wedge (d\alpha)^n$ is a volume form of $M$.
\end{definition}

\begin{definition} Let $(M, \alpha)$ be a contact manifold and $\phi: M \rightarrow M$ a diffeomorphism. We call $\phi$ a {\em contactomorphism} if $\phi^* \alpha = f \cdot \alpha$, where $f$ is a positive function on $M$. In the case $f=1$, $\phi$ is called a {\em strict contactomorphism}.
\end{definition}

\subsection{Mapping tori of contactomorphisms as LCS manifolds} We next describe a construction of compact LCS manifolds of rank 1 out of compact contact manifolds. The idea is to consider mapping tori of the latter by a contactomorphism.  

Indeed, let $(M, \al)$ be a compact contact manifold and $\phi: M \rightarrow M$ a contactomorphism. We define 
$$\bar{M}_{\phi}:=M \times [0,1]/(x, 0) \sim (\phi(x), 1).$$
We shall denote by $[m, t]$ a point of $\bar{M}_{\phi}$, which is the equivalence class of $(m, t)$ in $M \times [0,1]$. Then $\bar{M}_\phi$ has a natural structure of fiber bundle over $S^1$ with fiber $M$, given by $\pi: \bar{M}_{\phi} \rightarrow S^1$, $\pi([m, t])=e^{2\pi\mathrm{i}t}$. Here, $S^1$ is seen as the interval $[0, 1]$ with identified endpoints. 

Let $h$ be a smooth bump function, which is 0 near 0 and 1 near 1. On $M \times [0, 1]$, define the one-form $\tilde \al$ by:
$$\tilde{\alpha}:=\alpha(hf + (1-h)).$$
Then $\tilde{\alpha}$ descends to a one-form $\alpha_1$ on $\bar{M}_{\phi}$ which has the property that restricted to any fiber $M_t$ is a contact form. 

We denote by $\vartheta$ the volume form of the circle of length 1. The two-form $d\alpha_1 - \pi^* \vartheta \wedge \alpha_1$ is closed with respect to $d_{\pi^* \vartheta}$, but it is possible to have degeneracy points. Since $\bar{M}_\phi$ is compact, we may choose a large constant $K\gg 0$ such that 
$$\omega :=  d\alpha_1 - K \cdot \pi^* \vartheta \wedge \alpha_1$$ 
is nondegenerate. Moreover, we can choose $K$ in such a way that $e^K$ is not algebraic. Then $\omega$ defines an LCS form with the Lee form $\theta=K\pi^*\vartheta$. 

Clearly, the Lee form $\pi^*\vartheta$ is integral, and hence $\theta$ is a closed one-form of rank 1, with the  group of periods $K\Z$.

By the choice of $K$, the one-form $\theta$ is in the situation described by \ref{mail}, therefore the dimensions of the Morse-Novikov cohomology groups equal the Novikov Betti numbers of $\theta$. Since $\theta$ is nowhere vanishing, we obtain:
\begin{proposition}
Let $\bar M_\f$ and $\theta$ as above. Then the  Novikov Betti numbers vanish and the Morse-Novikov cohomology of $\theta$ is 0.
\end{proposition}

\begin{remark} Let us modify $\theta$ with an exact form $df$ such that $\theta_1 := \theta + df$ has at least one zero. By point (3) in \ref{prop1}, we get that $H_{\theta}(\bar{M}_{\phi})=H_{\theta_1}(\bar{M}_{\phi})=0$, hence the vanishing of the Morse-Novikov cohomology may occur also for forms which have zeros. In particular, the converse of the result in \cite{llmp} mentioned in Section 1 is not true. Namely, if Morse-Novikov cohomology vanishes, the one-form with respect to which it is considered may not be parallel, since a parallel one-form has no zeros.
\end{remark}

\begin{remark} The construction above describes a large class of LCS manifolds. The particular case when $\phi$ is the identity gives the product $M \times S^1$. In the case when $\phi$ is a strict contactomorphism, there is no need of choosing the bump function $h$, since $\alpha$ defines a global form on $\bar{M}_\phi$. Then a straightforward computation shows that $d\alpha - \pi^* \vartheta \wedge \alpha$ is nondegenerate, hence the constant $K$ may be chosen to be 1. This situation is described in \cite{b} and many examples are given in \cite{bm}.
\end{remark}

\subsection{The Inoue surface $S^0$} On the Inoue surfaces of type $\mathcal{S}^0$ we consider a closed-one form of rank one, for which $e^{\alpha}$ is algebraic. We present the explicit description of the Morse-Novikov cohomology groups. 

\subsubsection{Description of the LCS manifold $S^0$.} In  \cite{in}, M. Inoue introduced three types of complex compact surfaces, which are traditionally referred to as the Inoue surfaces $\mathcal{S}^0$, $\mathcal{S}^+$ and $\mathcal{S}^-$. In \cite{t}, Tricerri endowed the Inoue surfaces $S^0$, $S^+$ and some subclasses of $S^-$ with locally conformally K\" ahler metrics, in particular, by forgetting the complex structure,  with locally conformally symplectic structures. 

We are interested in the LCS structure on $S^{0}$ and we compute its corresponding Morse-Novikov cohomology. For this purpose, we review the construction of $\mathcal{S}^0$ and insist on its description as mapping torus of the 3-dimensional torus $\mathbb{T}^3$.

Let $A$ be a matrix from $\SL_3(\mathbb{Z})$ with one real eigenvalue $\alpha >1$ and two complex eigenvalues $\beta$ and $\bar{\beta}$. We denote by $(a_1,a_2,a_3)^t$  
 a real eigenvector of $\alpha$ and by $(b_1, b_2, b_3)^t$  
 a complex eigenvector of $\beta$. Let $G$ be the group of affine transformations of $\C \times \H$ generated by the transformations:
\begin{equation*}
\begin{split}
(z, w)& \mapsto (\beta z, \alpha w),\\
(z, w)& \mapsto (z+b_i, w+a_i).
\end{split}
\end{equation*}
for all $i=1, 2, 3$, where $\H$ stands for the Poincar\'e half-plane.

As a complex manifold,  $\mathcal{S}^0$ is  $(\C \times \H)/G$. 

We now explain  its structure as a mapping torus. Denote by $\mathbb{T}^3$ the standard 3-dimensional torus, namely $\mathbb{T}^3 = \R^3/ \langle f_1, f_2, f_3 \rangle$, where 
$f_1$ (resp.$f_2$, $f_3$) is the translation with $(1,0,0)$ (resp. $(0,1,0)$, $(0,0,1)$). 

Let $\bm{x}:=(x,y,z)^t$, and  consider the automorphism $\Phi$ of $\R^3$ with matrix $A^t$ in the canonical basis.
It clearly descends to an automorphism $\phi$ of $\mathbb{T}^3$, since $A^t$ belongs to $\SL_3(\Z)$. We define the manifold
$$\mathbb{T}^3 \times_{\phi}\R^+ : = (\mathbb{T}^3 \times \R^+)/(\widehat{ \bm{x}},t)\sim (\phi(\widehat{\bm{x}}),\al t)$$ 
which has the structure of a compact fiber bundle over $S^1$ by considering 
$$p : \mathbb{T}^3 \times_{\phi}\R^+ \rightarrow S^1,\qquad [(\widehat{ \bm{x}},t)]\mapsto e^{2\pi \mathrm{i} \mathrm{log}_{\alpha}t}$$
Here we denote by $[ , ]$ the equivalence class with respect to $\sim$.

In order to write explicitly a diffeomorphism between $\mathbb{T}^3 \times_{\phi}\R^+$ and $\mathcal{S}^0$, let
$$
B:=\left(\begin{smallmatrix}
    \Re b_1 & \Re_2 & \Re_3 \\
    \Im b_1 & \Im b_2 & \Im b_3\\
    a_1 & a_2 & a_3
     \end{smallmatrix} \right)
$$
Now the requested diffeomorphism acts as:
$$[\widehat{\bm{x}},t]\mapsto [[\widehat{B\cdot\bm{x}},t]],$$
where $x+ \mathrm{i}y$ and $z+\mathrm{i}t$ are coordinates on $\C \times \H$ and $\left[\left[x+\mathrm{i}y, z+\mathrm{i}t \right]\right]$ denotes the equivalence class of $\left(x+\mathrm{i}y, z+\mathrm{i}t\right)$. It is straightforward to check this map is well defined and indeed an isomorphism. 

The LCK structure given by Tricerri in \cite{t} is given as a $G$-invariant globally conformally K\"ahler structure on $\C \times \H$ and in the coordinates $(z, w)$, the expressions for the metric and the Lee form, respectively, are:
\begin{equation*}
\begin{split}
g&=-\mathrm{i}\frac{dw \otimes d\overline{w}}{w_2^2}+w_2dz \otimes d\bar{z}\\
\theta& = \frac{dw_2}{w_2},
\end{split}
\end{equation*}
where $w_2=\Im(w)$. For our description as fiber bundle and coordinates $(x, y, z, t)$, the Lee form $\theta$ is $\frac{dt}{t}$. 

As previously, we denote by $\vartheta$ the volume form of the circle of length 1. Obviously, the de Rham cohomology class $[\vartheta]$ is in $H^{1}(S^1, \Z)$ and implicitely $[p^*\vartheta]$ belongs to $H^1(\mathcal{S}^0, \Z)$, and hence the rank is 1. Moreover, a simple computation shows that 
$$\theta=\mathrm{ln}{\alpha} \cdot p^*\vartheta.$$
 So we are not in the situation depicted by \ref{mail}, since $e^{\mathrm{ln}\alpha}=\alpha$ is an algebraic integer.

\subsubsection{Explicit computation of the Morse-Novikov cohomology.} To compute by hand the Morse-Novikov cohomology groups of $\mathcal{S}^0$ with the twisted Mayer-Vietoris sequence, we first choose the open sets $U_1$ and $U_2$ which cover the circle:
\begin{equation*}
U_1 : =\{e^{2\pi\mathrm{i}t}\mid t \in (0, 1)\}, \qquad U_2:= \{e^{2\pi\mathrm{i}t}\mid t \in (\tfrac{1}{2},\tfrac{3}{2} )\},
\end{equation*}
and take as open sets $U:= p^{-1}(U_1)$ and $V:= p^{-1}(U_2)$, representing a covering of  $\mathcal{S}^0$. The sets $U$ and $V$ are the trivializations of  $\mathcal{S}^0$ as fiber bundle over $S^1$.  Therefore, we have
\begin{equation*}
\begin{split}
\f_{U_1}&:U{\longrightarrow}U_1\times\mathbb{T}^3,\qquad [w, t] \mapsto (e^{2\pi \mathrm{i}t}, w),\, t\in(1,\al),\\
\f_{U_2}&:V{\longrightarrow}U_2\times\mathbb{T}^3,\qquad [w, t] \mapsto (e^{2\pi \mathrm{i}t}, w),\, t\in(\alpha^{\frac{1}{2}}, \alpha^{\frac{3}{2}}).
\end{split}
\end{equation*}
Since $U_1 \cap U_2$ is disconnected,  the transition maps $g_{U_1U_2}:=\phi_{U_1} \circ \phi^{-1}_{U_2}$ are given by:
\begin{equation*}
\begin{split}
g_{U_1U_2}&: U_1 \cap U_2 \times \mathbb{T}^3 \rightarrow U_1 \cap U_2 \times \mathbb{T}^3,\\
g_{U_1U_2}(m,\bm{x}) 
&=\begin{cases}
(m,\widehat{\bm{x}}), & \text{if $m= e^{2\pi \mathrm{i}t}$, with $t \in (\frac{1}{2}, 1)$}\\[.5mm]
(m,\widehat{(A^t)^{-1}\cdot\bm{x}}), & \text{if $m= e^{2\pi \mathrm{i}t}$, with $t \in (1, \tfrac{3}{2})$}
\end{cases}
\end{split}
\end{equation*}


As $\theta$ is not exact, we already know that $H^0_{\theta}(\mathcal{S}^0)$ and $H^4_{\theta}(\mathcal{S}^0)$ vanish (see \cite{hr}). Concerning the other Morse-Novikov cohomolgy groups, we prove the following result:
\begin{theorem}\label{inoue} On $\mathcal{S}^0$, for the Lee form $\theta$ given by Tricerri, $H^{1}_{\theta}(\mathcal{S}^0)$ vanishes, $H^{2}_{\theta}(\mathcal{S}^0)\simeq \R$ and $H^{3}_{\theta}(\mathcal{S}^0)\simeq \R$.
\end{theorem}
\begin{proof} The proof is algebraic and the key is to explicitly write the morphism $\beta_*$. 

From now on we denote  by $W_1$ and $W_2$ the two connected components of $U_1 \cap U_2$, namely 
$$W_1=\left\{e^{2\pi \mathrm{i}t} \mid t \in (\tfrac{1}{2}, 1)\right\},\qquad  W_2=\left\{e^{2\pi \mathrm{i}t} \mid t \in  (1, \tfrac{3}{2})\right\}.$$
Consider the functions $ f : U_1 \rightarrow (0, 1)$, $ f(e^{2\pi \mathrm{i}t})=t$ and $\displaystyle g : U_2 \rightarrow (\tfrac{1}{2}, \tfrac{3}{2})$, $\displaystyle g(e^{2\pi \mathrm{i}t})=t$. Then on $U_1$, $\vartheta = df$ and on $U_2$, $\vartheta = dg$. Moreover, we observe that on $W_1$, $f$ and $g$ coincide and on $W_2$, $g=f+1$. Therefore, $\theta = \mathrm{ln} \alpha \cdot dp^*f$ on $U$ and $\theta = \mathrm{ln} \alpha \cdot dp^*g$ on $V$, hence $\theta$ is exact on these two open sets.

We have the following diagram:
\[
\begin{tikzcd}
H^0_{\theta_{|U}}(U) \oplus H^0_{\theta_{|V}}(V)\arrow{r}{\beta_*} \arrow[swap]{d}{\Phi} & H^0_{\theta_{|U \cap V}}(U \cap V) \arrow{d}{\Psi} \\
\R^2 \arrow{r}{\gamma} & \R^2
\end{tikzcd}
\]
where $\Phi$ and $\Psi$ are the isomorphisms defined as
\begin{equation*}
\begin{split}
\Phi ([\sigma], [\eta])& = (e^{-\mathrm{ln}\alpha p^*f}\sigma, e^{-\mathrm{ln}\alpha p^*g}\eta),\\
\Psi ([\omega])&= (e^{-\mathrm{ln}\alpha p^*f}\omega_{|p^{-1}(W_1)}, e^{-\mathrm{ln}\alpha p^*f}\omega_{|p^{-1}(W_2)}),
\end{split}
\end{equation*}
and $\gamma$ makes the diagram commutative, $\gamma(a, b)=(a-b, a- \alpha b)$. 

As $\alpha \neq 1$, $\gamma$ is an isomorphism, and hence  $\beta_*$ is an isomorphism, too. Consequently, the connecting morphism $\delta :  H^0_{\theta_{|U \cap V}}(U \cap V) \rightarrow H^1_{\theta}(\mathcal{S}^0)$ is injective and we can start the Mayer-Vietoris from $H^1_{\theta}(\mathcal{S}^0)$:
\begin{multline*}
0 \rightarrow H^1_{\theta}(\mathcal{S}^0) \rightarrow H^1_{\theta_{|U}}(U) \oplus H^1_{\theta_{|V}}(V) \rightarrow  H^1_{\theta_{|U \cap V}}(U \cap V)\rightarrow\\
\rightarrow \cdots  \rightarrow H^3_{\theta_{|U \cap V}}(U \cap V) \rightarrow 0
\end{multline*}

We look now at the other morphisms $\beta_*$ linking cohomology groups of degree $i \geq 1$.
\[
\begin{tikzcd}
H^i_{\theta_{|U}}(U) \oplus H^i_{\theta_{|V}}(V)\arrow{r}{\beta_*} \arrow[swap]{d}{\Phi} & H^i_{\theta_{|U \cap V}}(U \cap V) \arrow{d}{\Psi} \\
H^i_{dR}(\mathbb{T}^3) \oplus H^i_{dR}(\mathbb{T}^3) \arrow{r}{\gamma} & H^i_{dR}(\mathbb{T}^3) \oplus H^i_{dR}(\mathbb{T}^3)
\end{tikzcd}
\]
Using the fact that $\theta$ is exact when restricted to $U$ and $V$, the isomorphism $\Phi$ is obtained by the following composition of  isomorphisms:
$$H^i_{\theta_{|U}}(U) \stackrel{f_1}{\longrightarrow}H^{i}_{dR}(U)\stackrel{f_2}{\longrightarrow}
H^{i}_{dR}(U_1 \times \mathbb{T}^3)\stackrel{f_3}{\longrightarrow}H^{i}_{dR}(\mathbb{T}^3),
$$
where $f_1([\sigma])=[e^{-f}\sigma]$, $f_2([\eta])=[(\phi_{U_1})_{*}\eta]$, $f_3([\omega])=[i^*\omega]$ and $i:\mathbb{T}^3 \rightarrow U_1 \times \mathbb{T}^3$ is defined as $i(t) = (m, t)$, for some point $m$ in $U_1$. 

The same holds for $V$, the only difference being that $f_1':H^i_{\theta_{|V}}(V) \rightarrow  H^{i}_{dR}(V)$ is given by $[\sigma] \mapsto [e^{-g}\sigma]$ and $f_2': H^{i}_{dR}(V) \rightarrow H^{i}_{dR}(U_2 \times \mathbb{T}^3)$ is given by $[\eta] \mapsto [(\phi_{U_2})_*\eta]$. 
Thus:
$$\Phi = f_3 \circ f_2 \circ f_1 \oplus  f_3' \circ f_2' \circ f_1'.$$ 

As for $\Psi$, there is a similar sequence:
$$
H^i_{\theta_{|U \cap V}}(U \cap V)\stackrel{g_1}{\longrightarrow}H^{i}_{dR}(U \cap V)\stackrel{g_2}{\longrightarrow} H^{i}_{dR}(U \cap V \times \mathbb{T}^3)\stackrel{g_3}{\longrightarrow} H^i_{dR}(\mathbb{T}^3) \oplus H^i_{dR}(\mathbb{T}^3).
$$
Here, the isomorphisms $g_1$, $g_2$ and $g_3$ are given by $[\sigma] \mapsto [e^{-f}\sigma]$, $[\eta] \mapsto [(\phi_U)_*\eta]$ and $[\omega] \mapsto (i_1^*[\omega_{|W_1}], i_2^*[\omega_{|W_2}])$, where $i_1: \mathbb{T}^3 \rightarrow W_1 \times \mathbb{T}^3$ denotes the injection $t \mapsto (m, t)$ for some $m$ in $W_1$ and  $i_2: \mathbb{T}^3 \rightarrow W_2 \times \mathbb{T}^3$, $i_2(t) = (n, t)$ for some point $n$ in $W_2$. We define $\Psi = g_3 \circ g_2 \circ g_1$.  

A straightforward computation shows that $\gamma = \Phi^{-1} \circ \beta_* \circ \Psi$ is given by:
$$([a], [b]) \mapsto ([a -b], [a - \alpha \cdot i_2^*((g_{U_1U_2})_{|W_2})_*\pi^*b]),$$
where $\pi: V \times \mathbb{T}^3 \rightarrow \mathbb{T}^3$ is the projection on the second factor. 

We investigate now the map $i_2^*((g_{U_1U_2})_{|W_2})_*\pi^*: H^{i}_{dR}(\mathbb{T}^3) \rightarrow H^{i}_{dR}(\mathbb{T}^3)$ for $i=1, 2, 3$. It is an easy observation that $$i_2^*((g_{U_1U_2})_{|W_2})_*\pi^*=(\pi \circ (g_{U_1U_2})_{|W_2} \circ i_2)_*.$$

Since $\pi \circ (g_{U_1U_2})_{|W_2} \circ i_2 : \mathbb{T}^3 \rightarrow \mathbb{T}^3$ is given by the matrix $(A^t)^{-1}$, the map induced in homology, $(\pi \circ (g_{U_1U_2})_{|W_2} \circ i_2)_* : H_1(\mathbb{T}^3) \rightarrow H_1(\mathbb{T}^3)$ has the matrix $(A^{t})^{-1}$ in the canonical basis. Therefore, the matrix of the map induced by the pushforward $(\pi \circ (g_{U_1U_2})_{|W_2} \circ i_2)_* : H^1_{dR}(\mathbb{T}^3) \rightarrow H^1_{dR}(\mathbb{T}^3)$ in the canonical basis $\{[dx], [dy], [dz]\}$ is $(((A^{t})^{-1})^{t})^{-1}=A$. 

As a consequence, we obtain that the matrix of $\gamma : H^1_{dR}(\mathbb{T}^3) \oplus H^1_{dR}(\mathbb{T}^3) \rightarrow H^1_{dR}(\mathbb{T}^3) \oplus H^1_{dR}(\mathbb{T}^3)$ is the following:

$$
\left[
\begin{array}{c|c}
I_3 & -I_3 \\
\hline
I_3 & -\alpha \cdot A
\end{array}
\right]
$$

By performing a transformation which keeps the rank constant, namely adding the first three columns to the last three, we obtain that the aforementioned matrix has the same rank as:

$$
\left[
\begin{array}{c|c}
I_3 & O_3 \\
\hline
I_3 & I_3-\alpha \cdot A
\end{array}
\right]
$$

Moreover, this further implies that the rank is controlled by the block $I_{3} - \alpha \cdot A$, which would be a nonsingular matrix if and only if $\tfrac{1}{\alpha}$ were an eigenvalue of $A$, which is not the case. Hence, $\gamma$ and implicitly $\beta_*$ is an isomorphism, whence from the Mayer-Vietoris sequence, $H^{1}_{\theta}(\mathcal{S}^0)$ has to vanish. 

Since we already know the matrix of $(\pi \circ (g_{U_1U_2})_{|W_2} \circ i_2)_* : H^1_{dR}(\mathbb{T}^3) \rightarrow H^1_{dR}(\mathbb{T}^3)$ is $A$ in the basis $\{[dx], [dy], [dz]\}$, we  can easily compute the matrix of $(\pi \circ (g_{U_1U_2})_{|W_2} \circ i_2)_* : H^2_{dR}(\mathbb{T}^3) \rightarrow H^2_{dR}(\mathbb{T}^3)$ in the basis $\{[dy \wedge dz], [dz \wedge dx], [dx \wedge dz]\}$ to be $(A^*)^t$. Therefore, the matrix of $\gamma : H^2_{dR}(\mathbb{T}^3) \oplus H^2_{dR}(\mathbb{T}^3) \rightarrow H^2_{dR}(\mathbb{T}^3) \oplus H^2_{dR}(\mathbb{T}^3)$ is:

$$
\left[
\begin{array}{c|c}
I_3 & - I_3 \\
\hline
I_3 & -\alpha \cdot (A^*)^{t}
\end{array}
\right]
$$
which by the same arguments as above has the same rank as:
$$
\left[
\begin{array}{c|c}
I_3 & O_3 \\
\hline
I_3 & I_3-\alpha \cdot (A^*)^{t}
\end{array}
\right]
$$
Since $A^{*}=A^{-1} $(because $A$ lives in $\SL_3(\Z)$) and a matrix and its transpose have the same eigenvalues, $(A^*)^t$ has the same eigenvalues as $A^{-1}$, thus $\tfrac{1}{\alpha}$ is one of them. Therefore, the rank of the block $ I_3-\alpha \cdot (A^*)^{t}$ is 2, because $\tfrac{1}{\alpha}$ is an eigenvalue of $(A^*)^t$ of multiplicity 1. We infer that the matrix of $\gamma : H^2_{dR}(\mathbb{T}^3) \oplus H^2_{dR}(\mathbb{T}^3) \rightarrow H^2_{dR}(\mathbb{T}^3) \oplus H^2_{dR}(\mathbb{T}^3)$ has rank 5, forcing $\Ker\, \gamma$ to be 1-dimensional and from the Mayer-Vietoris sequence, we obtain $H^2_{\theta}(\mathcal{S}^0) \simeq \R$.

For the final case, when $i=3$, it is straightforward that $(\pi \circ (g_{U_1U_2})_{|W_2} \circ i_2)_* : H^3_{dR}(\mathbb{T}^3) \rightarrow H^3_{dR}(\mathbb{T}^3)$ is given by the multiplication with the determinant of the matrix of $(\pi \circ (g_{U_1U_2})_{|W_2} \circ i_2)_* : H^1_{dR}(\mathbb{T}^3) \rightarrow H^1_{dR}(\mathbb{T}^3)$. In this case, the determinant is 1, hence we get that $\gamma: H^3_{dR}(\mathbb{T}^3) \oplus H^3_{dR}(\mathbb{T}^3) \rightarrow H^3_{dR}(\mathbb{T}^3) \oplus H^3_{dR}(\mathbb{T}^3)$ is given by the $2 \times 2$ -matrix:
$$
\begin{bmatrix}
    1 & -1 \\
    1 & -\alpha\\
   
\end{bmatrix}
$$
and thus it defines an isomorphism. By the Mayer-Vietoris sequence, we obtain: 
$$\dim_{\R}H^3_{\theta}(\mathcal{S}^0)=6-\dim_{\R}\Im (\beta_*: H^2_{\theta_{|U}}(U) \oplus H^2_{\theta_{|V}}(V) \rightarrow H^2_{\theta_{|U \cap V}}(U \cap V))=1$$
In conclusion, $H^3_{\theta}(\mathcal{S}^0) \simeq \R$, $H^2_{\theta}(\mathcal{S}^0) \simeq \R$ and the rest of the Morse-Novikov cohomology groups vanish.
\hfill\end{proof}

\hfill

We now find generators for $H^2_{\theta}(\mathcal{S}^0)$ and $H^3_{\theta}(\mathcal{S}^0)$. 

Denote by 
$$\Omega :=-\mathrm{i}(\frac{dw \wedge d\overline{w}}{w_2^2}+w_2dz \wedge d\bar{z})$$
 the two-form on $\C \times \H$, in the coordinates $(z, w)$, which descends to an LCS form $\omega$ on $\mathcal{S}^0$. Notice that $\Omega_1:= -\mathrm{i}\frac{dw \wedge d\overline{w}}{w_2^2}$ and $\Omega_2: = -\mathrm{i}w_2dz \wedge d\bar{z}$ are two-forms which are invariant with respect to the factorization group $G$. They descend to $\mathcal{S}^0$ to two forms which we shall denote by $\omega_1$ and $\omega_2$ and we have $\omega=\omega_1+\omega_2$. Tricerri showed that $\omega$ is an LCK form and it is the fundamental two-form of the metric induced by $g=-\mathrm{i}\frac{dw \otimes d\overline{w}}{w_2^2}+w_2dz \otimes d\bar{z}$ on $\mathcal{S}^0$, which we shall denote by $g_1$. Then we have the following:

\begin{proposition}\label{gen} Let $\omega$ be the above defined LCS form of $\mathcal{S}^0$ and $\theta =\tfrac{dw_2}{w_2}$ its Lee form, as in \ref{inoue}. Then:
\begin{equation*}
\begin{split}
H^{2}_\theta(\mathcal{S}^0)& = \R [\omega]\\
H^{3}_\theta(\mathcal{S}^0)& = \R [\theta \wedge \omega].
\end{split}
\end{equation*}
\end{proposition}

Before proving these equalities, we define the notion of {\em twisted laplacian}. Namely, by extending the metric $g_1$ to the space of $k$-forms $\Omega^k(\mathcal{S}^0)$, we consider the Hodge star operator $*: \Omega^k(\mathcal{S}^0) \rightarrow \Omega^{4-k}(\mathcal{S}^0)$, given by $u \wedge * v=g_1(u, v) d\vol$. Note that the real dimension of $\mathcal{S}^0$ is 4.
Then the following operators depending on $\theta$ can be defined (they indeed  make sense on any manifold $M$ endowed with a closed one-form $\theta$, although we shall treat specifically the case of $\mathcal{S}^0$):

\begin{equation*}
\begin{split}
\delta_\theta:&\, \Omega^{k+1}(\mathcal{S}^0) \rightarrow \Omega^{k}(\mathcal{S}^0), \qquad 
\delta_\theta =- *d_{-\theta} *\\
\Delta_\theta:&\, \Omega^k(\mathcal{S}^0) \rightarrow \Omega^{k}(\mathcal{S}^0),\qquad
\Delta_\theta = \delta_\theta d_\theta + d_\theta \delta_\theta
\end{split}
\end{equation*}

\begin{remark} $\delta_\theta$ is the adjoint of $d_\theta$ with respect to the inner product on $ \Omega^k(\mathcal{S}^0)$ given by $\langle \eta, \phi \rangle =\int_{\mathcal{S}^0}\eta \wedge * \phi$.  Observe that $\delta_\theta$ and $\Delta_\theta$ are perturbations of the usual operators codifferential and laplacian, which are recovered by replacing $\theta$ with 0. The motivation for introducing the operators twisted with $\theta$ is to develop Hodge theory in the context of working with $d_\theta$ instead of $d$. They were first considered in \cite{va1} in the context of locally conformally  K\"ahler manifolds and later in \cite{gl} in the LCS setting.
\end{remark}

The following analogue of Hodge decomposition holds:
\begin{theorem} {\rm ( \cite{gl})} Let $M$ be a compact manifold, $\theta$ a closed one-form, $\delta_\theta$ and $\Delta_\theta$ defined as above. Then we have an orthogonal decomposition:
\begin{equation}\label{hodge}
\Omega^{k}(M)=\mathcal{H}^k_\theta(M) \oplus d_\theta \Omega^{k-1}(M) \oplus \delta_\theta\Omega^{k+1}(M)
\end{equation}
where $\mathcal{H}^k_\theta(M)=\{\eta \in  \Omega^{k}(M) \mid \Delta_\theta \eta=0\}$.  Moreover,
$$H^{k}_\theta(M) \simeq \mathcal{H}^k_\theta(M).$$
\end{theorem}
Thus, we observe that important properties of the Hodge-de-Rham theory for the operator $d$ are shared by the same theory applied to $d_\theta$. 

We  now give the\\[1mm]
\noindent{\em Proof of \ref{gen}.} 
Since we proved in \ref{inoue} that $H^{2}_\theta$ and $H^{3}_\theta$ are isomorphic to $\R$, it is enough to show that $\omega$ and $\theta \wedge \omega$ are $d_{\theta}$-closed, but not $d_{\theta}$-exact.  

We shall prove that with respect to the Hodge decomposition \eqref{hodge}, $\omega$ has the harmonic and the $d_{\theta}$-exact parts non-vanishing. Indeed, a straightforward computation shows that $\Omega_1=d_{\tfrac{dw_2}{w_2}}\tfrac{-dw_1}{w_2}$. Since $\tfrac{-dw_1}{w_2}$ is $G$-invariant and descends to a one-form $\eta$ on $\mathcal{S}^0$, we have $w_1 = d_{\theta}\eta$. As $\omega$ is the fundamental two-form of the metric $g_1$, which is hermitian with respect to the complex structure of $\mathcal{S}^0$ induced form the standard one on $\C \times \H$, an easy linear algebra computation (see \cite[p. 31]{gh}) shows that the Riemannian volume form $d\vol$ equals $\tfrac{\omega^2}{2!}$. In the general case of complex dimension $n$, the volume form $d\vol$ is $\tfrac{\omega^n}{n!}$. This further implies that $*\omega_2 = \omega_1$. Consequently, $d_{-\theta} * \omega_2 = d_{-\theta}\omega_1 =d{\omega_1}+\theta \wedge \omega_1$. However, $d\Omega_1=0$ and $\tfrac{dw_2}{w_2} \wedge \Omega_1=0$,  hence $d\omega_1=0$ and $\theta \wedge \omega_1=0$, implying that $\omega_2$ is $\delta_\theta$-closed. Still, one can show that $\Omega_2$ is $d_{\tfrac{dw_2}{w_2}}$-closed, therefore $\omega_2$ also is $d_{\theta}$-closed. So $\omega_2$ is harmonic with respect to $\Delta_\theta$. Thus, $\omega=\omega_1+\omega_2$ is the Hodge decomposition of $\omega$. We proved in this way that $\omega$ is not $d_\theta$-exact and moreover, $[\omega]=[\omega_2]$ defines a non-vanishing cohomology class in $H^{2}_\theta(M)$. But $H^{2}_\theta(M) \simeq \R$, therefore $H^{2}_\theta(M)= \R[\omega]=\R[\omega_2]$.

As for $H^{3}_\theta(M)$, we first notice that $\theta \wedge \omega$ is $d_\theta$-closed. Indeed, $d_{\theta} (\theta \wedge \omega)=d(d\omega)-\theta \wedge \theta \wedge \omega =0$. In \cite{g}, it was shown that $\Delta_\theta (\theta \wedge \omega) =0$, whence we obtain, as in the case of $\omega$, that $\theta \wedge \omega$ is not $d_\theta$-exact. This means that we found a generator for $H^{3}_\theta(M)$, namely $H^{3}_\theta(M)=\R[\theta \wedge \omega]$. \hfill\blacksquare

\begin{remark} We notice that the alternate sum of the dimensions of the Morse-Novikov cohomology $H^{i}_\theta(\mathcal{S}^0)$ groups is 0, which equals indeed the Euler characteristic of $\mathcal{S}^0$. 
\end{remark}

\noindent{\bf Acknowledgements:} I am very grateful to Alexandru Oancea for the original suggestion of the theme and to Liviu Ornea for his encouragement and valuable ideas and suggestions that improved this paper. Many thanks to Miron Stanciu are due for very enlightening discussions. I also thank Andrei Pajitnov for drawing my attention to the results in the paper \cite{paj}.

\end{document}